\documentclass[11pt]{amsart}
\usepackage{latexsym}
\usepackage{amsfonts,amssymb,amsmath,amsthm,amscd}
\usepackage[mathscr]{eucal}

\newcommand{\I}{\mbox{${\mathbb I}$}}

\newcommand{\R}{\mbox{${\mathbb R}$}}

\newcommand{\tr}{{\rm tr}}
\newcommand{\ric}{{\rm Ric}}

\makeatletter
\def\numberwithin#1#2{\@ifundefined{c@#1}{\@nocnterrr}{%
  \@ifundefined{c@#2}{\@nocnterr}{%
  \@addtoreset{#1}{#2}%
  \toks@\expandafter\expandafter\expandafter{\csname the#1\endcsname}%
  \expandafter\xdef\csname the#1\endcsname
    {\expandafter\noexpand\csname the#2\endcsname
     .\the\toks@}}}}
\makeatother
\numberwithin{equation}{section}

\newtheorem{thm}[equation]{Theorem}
\newtheorem{lemma}[equation]{Lemma}
\newtheorem{prop}[equation]{Proposition}

\newtheorem{cor}[equation]{Corollary}
\newtheorem{ex}[equation]{Example}

\newtheorem{rem}[equation]{Remark}
\newenvironment{rmk}{\begin{rem} \em}{\end{rem}}

\begin{document}

\title{Some New Examples of Non-K\"ahler Ricci Solitons}
\author{Andrew S. Dancer}
\address{Jesus College, Oxford University, OX1 3DW, United Kingdom}
\email{dancer@maths.ox.ac.uk}
\author{McKenzie Y. Wang}
\address{Department of Mathematics and Statistics, McMaster
University, Hamilton, Ontario, L8S 4K1, Canada}
\email{wang@mcmaster.ca}
\thanks{partially supported by NSERC Grant No. OPG0009421}

\date{revised \today}

\begin{abstract}
We produce non-K\"ahler complete steady gradient Ricci solitons
generalising those constructed by Bryant and Ivey.
\end{abstract}

\maketitle


\bigskip
\setcounter{section}{-1}

\section{\bf Introduction}

In this article we continue our investigation of reductions of the Ricci soliton
equations to ordinary differential equations. Recall that a Ricci soliton consists of
a complete Riemannian metric $g$ and a complete vector field $X$ on a manifold
satisfying the equation:
\begin{equation} \label{LieRS}
{\ric}(g) + \frac{1}{2} \,{\sf L}_{X} g + \frac{\epsilon}{2} \, g = 0
\end{equation}
where $\epsilon$ is a real constant and $\sf L$ denotes the Lie derivative.

This equation is a generalisation of the Einstein equation, and it is
natural to look for solutions by methods that have been fruitful in the Einstein
case. In \cite{DW} we set up the formalism for cohomogeneity one Ricci solitons
and wrote down the resulting ODE system. We found families of explicit K\"ahler
solutions generalising those of \cite{Ko}, \cite{Ca}, \cite{ChV},
\cite{G}, \cite{PTV}, \cite{FIK}, \cite{PS} and \cite{ACGT}.

Most of the known examples of Ricci solitons are indeed K\"ahler. The
exceptions of which we are aware are the homogeneous solitons on nilpotent
Lie groups \cite{La}, the rotationally symmetric Bryant
solitons \cite{Bry} on $\R^n \, (n > 2),$   Ivey's generalization of these
solutions \cite{Iv}, as well as the expanding counterparts described in \cite{Cetc}
and \cite{GK}. (Note that if $n=2$ the Bryant soliton is Hamilton's
famous cigar soliton \cite{Ha1}, which {\em is}  K\"ahler.) The Bryant solutions
are warped products on a single factor, while those of Ivey involve two factors.

In this paper we shall focus on steady gradient Ricci solitons and
generalise the Bryant-Ivey examples to produce complete steady  solitons
on warped products over an arbitrary number of positive Einstein factors
(see Theorem \ref{mainthm}). An important tool in our analysis is the
observation that the {\em general} cohomogeneity one {\em steady} soliton equations
always admit a Lyapunov function. This generalises the Lyapunov function in
the Bryant-Ivey systems.

Recall that Hamilton has proved that steady gradient Ricci solitons occur as
type II singularity models for the Ricci flow when the curvature operator is
non-negative and the Ricci is positive (\cite{Ha2} or Theorem 4.3.6 in \cite{CaZ}).
Our new steady solitons have non-negative Ricci curvature (Theorem \ref{ricci})
and always some negative sectional curvatures. They have asymptotically
paraboloid geometry (Theorem \ref{asymptotics}) and hence zero asymptotic
volume ratio. The asymptotic scalar curvature ratio is infinite
(cf Remark \ref{sectcurv}). Whether these steady solitons can be realised
as blow-up limits of non-trivial Ricci flows seems to be an interesting
question.

Finally, we would like to thank the referee for several helpful comments.

\section{\bf Lyapunov functions}

We recall the set-up from \cite{DW}. We consider a manifold $M$
with an open dense set foliated by diffeomorphic hypersurfaces $P_t$ of
real dimension $n$.
Assume the metric can be written as $g=dt^2 + g_t$ where $g_t$ is a metric
on $P_t$. We can view $t$ as arclength along a geodesic orthogonal to the
hypersurfaces. Let $r_t$ denote the Ricci tensor of $g_t$, viewed as an
endomorphism via $g_t$, and let $L_t$ denote the shape operator of the
hypersurfaces (so $\dot{g_t} = 2 g_t L_t$). Assume that the scalar curvature
$R_t=\tr(r_t)$ and the mean curvature $\tr(L_t)$ are constant on each
hypersurface. Furthermore, assume that the codifferentials $\delta^{\nabla^{t}} L_t$
vanish, where $L_t$ is viewed as a $TP_t$-valued $1$-form on $P_t$.

The above assumptions are satisfied, for example, if $M$ is
of cohomogeneity one with respect to an isometric group action, with
no repeated summands in the isotropy representation of the principal orbits $P_t$.
They are satisfied also when $M$ is a multiple warped product over an
interval, which will be the setting of this paper.

We consider solitons of {\em gradient type}, i.e., where $X = {\rm grad} \; u$
for a function $u$. Equation (\ref{LieRS}) then becomes
\begin{equation} \label{gradRS}
{\ric}(g) + {\rm Hess}(u) + \frac{\epsilon}{2} \, g = 0.
\end{equation}
We will further suppose that $u$ is a function of $t$ only. In this setting,
the above equation become the system (cf \S 1 of \cite{DW})
\begin{eqnarray}
  -{\rm tr} (\dot{L}) - {\rm tr} (L^2) + \ddot{u} + \frac{\epsilon}{2} &=& 0,  \label{TT} \\
    r_t - ({\rm tr}\, L)L - \dot{L} + \dot{u} L + \frac{\epsilon}{2} \, \I &=& 0 \label{SS}.
\end{eqnarray}

We have a conservation law
\[
\ddot{u} + ({\rm tr}\, L) \dot{u} - (\dot{u})^2 -\epsilon u = C
\]
for some constant $C$.
Using the equations this may be rewritten as
\begin{equation}
{\rm tr} (r_t) + {\rm tr} (L^2) - (\dot{u} - {\rm tr}\, L )^2 - \epsilon u +
\frac{1}{2}(n-1) \epsilon = C.
\end{equation}

We now specialise to the case of {\em steady solitons}, that is, $\epsilon =0$.
The conservation law is now
\begin{equation} \label{cons1}
{\rm tr} (r_t) + {\rm tr} (L^2) - (\dot{u} - {\rm tr}\,  L)^2 = C.
\end{equation}

\begin{prop}
The function $(\dot{u} - {\rm tr}\, L)^{-2}$
is a Lyapunov function, that is, it is monotonic on each interval on which it is defined.
\end{prop}

\begin{proof}
\begin{eqnarray*}
\frac{d}{dt} \left( \frac{1}{(\dot{u} - {\rm tr}\, L)^2} \right) &=&
-\frac{2 ( \ddot{u} - {\rm tr}\, \dot{L} )}{(\dot{u} - {\rm tr}\, L)^3} \\
&=& - \frac{2 {\rm tr} (L^2) }{(\dot{u} - {\rm tr}\, L)^3}
\end{eqnarray*}
\end{proof}

\begin{rmk}
The conservation law (\ref{cons1}) shows that our Lyapunov function is a constant multiple
of
\[
\frac{{\rm tr}(r_t) + {\rm tr} (L^2)}{(\dot{u} - {\rm tr}\; L)^2}  -1.
\]
\end{rmk}

\begin{rmk} Recall that the above conservation law was derived in \cite{DW}
from the consequence $\Delta du = du \circ \ric$ of the Ricci soliton
equation upon application of the contracted second Bianchi identity.
In fact, we can also derive this conservation law from Perelman's
$\mathcal F$-functional (steady case) and $\mathcal W$-functional
(expanding and contracting cases). From this point of view, the conservation
law asserts that the trajectories representing smooth gradient Ricci
solitons must lie in the zero-energy hypersurface of the Hamiltonians
corresponding to Perelman's functionals. Details of the Hamiltonian
formulation will be discussed elsewhere.
\end{rmk}

\section{\bf Multiple warped products}
We now specialise to the class of examples, multiple warped products,
that will generalise the examples of Bryant and Ivey.

We look for metrics of the form
\begin{equation} \label{metric}
dt^2 + \sum_{i=1}^{r} g_i^2(t)\,  h_i
\end{equation}
on $I \times M_1 \times ... \times M_r$\, where
$I$ is an interval in $\mathbb R$ and $(M_i, h_i)$ are Einstein manifolds
with positive Einstein constants $\lambda_i$. We let $d_i$ denote the (real)
dimension of $M_i$.

Recall that the soliton potential $u$ is taken to be a function of $t$ alone.
The resulting equations are equivalent to those coming from a cohomogeneity one
ansatz, though of course the $M_i$ could be inhomogeneous.

The shape operator and Ricci endomorphism are now given by
\begin{eqnarray*}
L_t &=& {\rm diag} \left( \frac{\dot{g_1}}{g_1} \,\I_{d_1}, \cdots,\frac{\dot{g_r}}{g_r} \,\I_{d_r}\right)  \\
r_t &=& {\rm diag} \left( \frac{\lambda_1}{g_1^2} \, \I_{d_1}, \cdots,\frac{\lambda_r}{g_r^2}\, \I_{d_r}\right)
\end{eqnarray*}
where $\I_m$ denotes the identity matrix of size $m$. (We shall henceforth
drop the subscript $t$ of $L$ for ease of notation.)
Motivated by Ivey's work, we introduce new variables
\begin{eqnarray}
X_i &=& \frac{\sqrt{d_i}}{( -\dot{u} + {\tr} L)} \frac{\dot{g_i}}{g_i} \label{def-Xi}\\
Y_i &=& \frac{\sqrt{d_i \lambda_i}}{g_i} \frac{1}{(-\dot{u} + \tr{L})} \label{def-Yi}
\end{eqnarray}
for $i=1, \ldots, r$.

Notice that
\[
\sum_{j=1}^{r} \, X_j^2  = \frac{{\rm tr} (L^2)}{(\dot{u} - {\rm tr} L)^2} \;\; : \;\;
\sum_{j=1}^{r} \, Y_j^2 = \frac{ {\rm tr} (r_t)}{(\dot{u} - {\rm tr} L)^2}.
\]
We can take our Lyapunov function, therefore, to be
\begin{equation} \label{Lyap}
{\mathcal L} := \frac{C}{(\dot{u}- {\rm tr} L)^2} = \sum_{i=1}^{r} \, (X_i^2 + Y_i^2)  - 1
\end{equation}
where $C$ is a nonzero constant to be specified later. Note further that
\begin{equation} \label{cons2}
{\mathcal L} = \tilde{C}_{i} g_i^2 Y_i^2 \;\; : \;\; i=1, \ldots, r
\end{equation}
for nonzero constants $\tilde{C}_{i} = \frac{C}{d_i \lambda_i}$.

It is convenient to introduce a new independent variable $s$ defined by
\begin{equation} \label{st}
\frac{d}{ds} := \frac{1}{(-\dot{u} + {\rm tr} L)} \frac{d}{dt} =
\sqrt{\frac{\mathcal L}{C}} \frac{d}{dt}
=\frac{g_i Y_i}{\sqrt{\lambda_i d_i}}\, \frac{d}{dt},
\end{equation}
in which the final expression is independent of $i$ by (\ref{cons2}).
We use a prime ${ }^{\prime}$ to denote differentiation with respect to $s$.

We obtain from the Ricci soliton system the following equations in our new variables:
\begin{eqnarray}
X_{i}^{\prime} &=& X_i \left( \sum_{j=1}^{r} X_j^2 -1 \right) +
\frac{Y_i^2}{\sqrt{d_i}} \, , \label{eqnX} \\
Y_{i}^{\prime} &=&  Y_i \left( \sum_{j=1}^{r} X_j^2 -\frac{X_i}{\sqrt{d_i}}
\right) \label{eqnY}
\end{eqnarray}
for $i=1, \ldots, r$.   Note that these imply the equation
\begin{equation} \label{eqnLyap}
{\mathcal L}^{\prime} = 2 {\mathcal L} \left( \sum_{i=1}^{r} X_i^2 \right).
\end{equation}

Conversely, if we have a solution of the above system, we may recover $t$ and the $g_i$
from
\begin{equation} \label{def-tgi}
dt = \sqrt{\frac{\mathcal L}{C}}\,\, ds, \ \ \ \ \
{g}_i = \sqrt{\frac{\mathcal L}{C}} \, \frac{\sqrt{d_i \lambda_i}}{Y_i}\, ,
\end{equation}
which are equivalent to (\ref{st}) and (\ref{def-Yi}) respectively. The soliton potential
is recovered from integrating
\begin{equation} \label{def-u}
 \dot{u} = \tr(L) - \sqrt{\frac{C}{\mathcal L}},
\end{equation}
where $\tr(L)$ is calculated using
\begin{equation} \label{logdiff-gi}
 \frac{\dot{g_i}}{g_i} = \sqrt{\frac{C}{\mathcal L}} \frac{X_i}{\sqrt{d_i}}\, .
\end{equation}
Differentiating the above, which is equivalent to (\ref{def-Xi}), one gets
\begin{equation} \label{gidotdot}
     \frac{\ddot{g}_i}{g_i} = \frac{C}{\mathcal L} \left(\frac{X_i^2 + Y_i^2 -\sqrt{d_i} X_i}{d_i}\right).
\end{equation}
Putting (\ref{def-u})-(\ref{gidotdot}) together one obtains (\ref{SS}). Finally,
differentiating (\ref{def-u}) gives (\ref{TT}). We then obtain a smooth solution of the
Ricci soliton equation provided that appropriate smoothness conditions
at the endpoints (formulated in the next section) are satisfied.

\begin{rmk} \label{einstein}
In (\ref{def-u}) if we take the derivative with respect to $s$ instead, we have
$$ u^{\prime} = \frac{\tr L}{\sqrt{C/{\mathcal L}} } - 1 = \sum_{i=1}^{r} \sqrt{d_i} X_i - 1,$$
where we have used (\ref{logdiff-gi}). Observe that $u^{\prime} \equiv 0$ iff
the soliton is trivial (Ricci-flat in the steady case). This motivates the
definition of the quantity $\mathcal{H} := \sum_{i=1}^{r} \sqrt{d_i} X_i$,
so that the Ricci-flat trajectories lie in the subspace ${\mathcal H} = 1$.
For these trajectories the conservation law (\ref{cons1})
becomes $\mathcal{L} = 0$.
\end{rmk}

\section{\bf Trajectories of the equations}

Recall that by applying the maximum principle to consequences of
the Ricci soliton equation one can show that on a closed manifold
a steady Ricci soliton is Ricci-flat (see e.g. Proposition 1.66 in \cite{Cetc}).
Hence we are interested in constructing complete non-compact steady
soliton metrics where there is a smooth collapse at one end, corresponding to $t=0$
without any loss of generality, onto a lower-dimensional submanifold. This can
be achieved if we take one factor, say $M_1$, to be a sphere $S^{d_1}$.
The submanifold would then be $M_2 \times \cdots \times M_r$. With
the normalization $\lambda_1 = d_1 -1$, the boundary conditions for the soliton
solution to be $C^2$ are the existence of the following limits:
\begin{equation} \label{bdy0}
g_1(0)=0   \;\; : \;\; g_i(0) = l_i \neq 0 \; (i > 1),
\end{equation}

\begin{equation} \label{bdy1}
\dot{g_1}(0)=1 \;\; : \;\; \dot{g_i}(0) = 0 \; (i > 1),
\end{equation}

\begin{equation} \label{bdy2}
\ddot{g}_{1}(0) =0 \;\; : \;\; \ddot{g_i}(0) \;\, {\rm  finite} \;
(i > 1),
\end{equation}

\begin{equation} \label{bdyu}
u(0) \, \,  {\rm finite}: \, \, \dot{u}(0)=0 \, \, : \ddot{u}(0) \, \, {\rm \, finite}.
\end{equation}

In order to get a smooth solution, it suffices to show further that the third
derivatives of $g_i$  tend to finite limits at $t=0$. Once this is done,
we can write the trace of the soliton equation (\ref{gradRS}) as
$$ \Delta u =  R + \frac{n \epsilon}{2}.$$
The right-hand side of this elliptic equation for $u$ lies in $C^{0, \alpha}$.
Since $u$ is in $C^2$, it follows from Lemma 6.16 in \cite{GT} that $u$
is in $C^{2, \alpha}$. Using the contracted second Bianchi identity and the
weak form of Bochner's formula for the Laplacian of a one-form (verified for
example by smooth approximation in $W^{1, 2}$) we can then show that the $1$-form
$\omega:=du$ is a weak solution of $\Delta \omega = 2 \omega \circ r$
(Eq. (2.1) in \cite{DW}). The argument in the proof of Lemma 2.2
there shows that $\omega$ is actually in $C^{2, \alpha}$. The smoothness
(in fact real analyticity) of the solution then follows from Morrey's theorem.

With the above remarks in mind, we consider trajectories emanating from the
critical point of (\ref{eqnX}) and (\ref{eqnY}) given by
\[
X_i = \beta, \;\; Y_1 = \hat{\beta}, \;\; X_i =Y_i=0  \,\, (i > 1)
\]
where $\beta = \frac{1}{\sqrt{d_1}} $ and $\hat{\beta} = +\sqrt{1- \beta^2}$.
This critical point lies on the unit sphere in $XY$-space, i.e., on the
level set  ${\mathcal L}=0$.

Linearising about this critical point gives a system whose matrix
has a $2 \times 2$ block
\[
\left( \begin{array}{cc}
3 \beta^2 -1 & 2 \beta \hat{\beta} \\
\beta \hat{\beta} & 0
\end{array} \right)
\]
corresponding to $X_1, Y_1$ : the remaining entries are diagonal, with
$\beta^2$ and $\beta^2 -1$ each occurring $r-1$ times.
The eigenvalues are therefore
$\beta^2$ ($r-1$ times),  $\beta^2 -1$ ($r$ times), and $2 \beta^2$.

\medskip

We shall assume from now on that $d_1 > 1$. The above critical point
is then {\em hyperbolic}.

\medskip

We will parametrise trajectories emanating from this critical point
so that the critical point corresponds to $s=-\infty$. Note that there is
some $\delta > 0$  such that for $i>1$, the differential inequality
$$  \frac{d}{ds} Y_i^2 \leq 2(\beta^2 + \delta) Y_i^2$$
holds near $s=-\infty$ for any such trajectory. A comparison argument then
shows that if $Y_i(s_*) > 0$ for all $i>1$, then on $(-\infty, s_*]$ we have
$Y_i > 0$. Since $\lim_{s \rightarrow -\infty} Y_1 = \hat{\beta} > 0$ we may
assume that $Y_1 > 0$ on $(-\infty, s_*]$ as well.

Now by standard facts in dynamical systems (cf \cite{CL}, proofs of
Theorems 4.1, 4.3 and 4.5) and the fact that the system
(\ref{eqnX})-(\ref{eqnY}) is invariant under the symmetries
$Y_i \mapsto -Y_i$, there is an $(r-1)$-parameter family of trajectories
lying in the unstable manifold of this critical point having the above
positivity properties and flowing into the open unit ball ${\mathcal L} < 0$.

\begin{rmk} \label{Bohm}
In \cite{Bo} an $r-2$ parameter family of complete Ricci-flat
metrics was constructed on the manifolds under consideration here.
These correspond to trajectories emanating from the above critical point and
lying {\em in} the sphere ${\mathcal L} = 0$. In fact, the unstable manifold
intersects this sphere transversely, and this accounts for the parameters
in the Ricci-flat metrics.
\end{rmk}

We now work with one of the trajectories going into ${\mathcal L} < 0$.

 Eq. (\ref{eqnLyap}) shows
that the trajectory {\em stays} in the region ${\mathcal L} < 0$. Hence all
the variables are bounded by $1$ and the flow exists for all $s \in \mathbb R$.
Moreover, $\mathcal L$ decreases monotonically to some negative constant $\kappa$.
Note that as a result of the above choices, the constant $C$ in (\ref{Lyap})
is fixed and is {\em negative}.

Now  Eq. (\ref{eqnY}) and $X_i \leq 1$ imply that $Y_i^2$ satisfies
the differential inequality
$$ \frac{d}{ds} Y_i^2 \geq -\frac{2}{\sqrt{d_i}} \, Y_i^2. $$
By a standard comparison argument it follows from $Y_i(s_*) > 0$ that
$Y_i > 0$ on $[s_*, +\infty)$. Hence for all finite $s$, $g_i$ can be defined
by (\ref{def-tgi}) (so is nonzero) and (\ref{cons2}) holds.

\begin{lemma}
The metric corresponding to our trajectory is complete at $s = +\infty$.
\end{lemma}

\begin{proof}
From (\ref{def-tgi}) it follows that as $s$ tends to infinity the arclength
$t$ does also. The definition of $g_i$ shows that $g_i$ remain nonzero
on $s > s_*$.
\end{proof}

We can refine our analysis to study the asymptotics of the metric as
$s$ tends to $+\infty$.

\begin{prop} \label{omegalimit}
The trajectory converges to the origin as $s$ tends to $+\infty$.
\end{prop}

\begin{proof}
Recall that the $\omega$-limit set of the trajectory is the set
\[
\Omega = \{ (X^*,Y^*) : \exists \, s_n \rightarrow +\infty \; {\rm  with} \;
(X(s_n),Y(s_n)) \rightarrow (X^*,Y^*) \}.
\]
As our trajectory ultimately lies in a compact set, we know from standard theory
(\cite{Pk} \S 3.2) that $\Omega$ is a compact, connected, non-empty set that is
invariant under the flow of our equations. Moreover, $\Omega$ is contained in
the sphere ${\mathcal L}= \kappa$.

Now if $\Omega$ contains a point $(X^*, Y^*)$ with $X^* \neq 0$, we see
from (\ref{eqnLyap}) that ${\mathcal L}^{\prime}$
at this point is nonzero, contradicting the flow-invariance of $\Omega$.
Hence $\Omega$ is contained in the set $X_i=0 \; (i=1, \ldots, r)$.
 Furthermore, if
$\Omega$ contains a point $(0, Y^*)$ with $Y^* \neq 0$, we see from (\ref{eqnX}) that
some $X_i^{\prime}$ is nonzero, again contradicting flow-invariance.

Hence $\Omega = (0,0)$, and the limiting value $\kappa$ of $\mathcal L$ is
$-1$, showing that the trajectory does indeed converge to the origin.
\end{proof}

\begin{lemma}
We have $\lim_{s \rightarrow \infty} \frac{X_i}{Y_i^2} = \frac{1}{\sqrt{d_i}}$.
\end{lemma}

\begin{proof}
Observe that $X_i/Y_i^2$ satisfies the differential equation
\begin{equation} \label{eqnXY2}
\left( \frac{X_i}{Y_i^2} \right)^{\prime} =  \left(-1 -\sum_{j=1}^{r} X_j^2
+ \frac{2 X_i}{\sqrt{d_i}} \right) \frac{X_i}{Y_i^2} + \frac{1}{\sqrt{d_i}}\, .
\end{equation}
By Prop \ref{omegalimit}, the coefficient of $\frac{X_i}{Y_i^2}$ tends
to $-1$ as $s \rightarrow +\infty$. In particular, if $\frac{X_i}{Y_i^2}$
tends to a limit $a,$ then its derivative tends to
$-a + \frac{1}{\sqrt{d_i}},$ so $a$ must equal $\frac{1}{\sqrt{d_i}}$.

Let $0 < \delta < 1$, and pick $s^*(\delta)$ so that
the absolute value of $-\sum_{j=1}^{r} X_j^2 +
\frac{2X_i}{\sqrt{d_i}}$ is less than $\delta$ for $s > s^*(\delta)$.
It follows that
if $\frac{X_i}{Y_i^2} (s_0) \geq \frac{1}{\sqrt{d_i}(1 - \delta)}$
for some $s_0 > s^*(\delta),$
then $\left(\frac{X_i}{Y_i^2} \right)^{\prime} <0$ at $s_0$.
Similarly, if $\frac{X_i}{Y_i^2}(s_0) \leq \frac{1}{\sqrt{d_i}(1+\delta)}$,
 then $\left(\frac{X_i}{Y_i^2} \right)^{\prime} >0$ at $s_0$.

So if $\frac{X_i}{Y_i^2}$ enters the horizontal strip
$\frac{1}{\sqrt{d_i}(1+\delta)} < y < \frac{1}{\sqrt{d_i}(1-\delta)}$
at some $s > s^*(\delta)$ it is trapped there.
Hence one of the following must hold:

(i) $\frac{X_i}{Y_i^2} (s) \geq \frac{1}{\sqrt{d_i}(1-\delta)}$ for all $s > s^*(\delta),$

(ii) $\frac{X_i}{Y_i^2} (s) \leq \frac{1}{\sqrt{d_i}(1+\delta)}$ for all $s > s^*(\delta),$
or

(iii)
 $ \frac{1}{\sqrt{d_i}(1+\delta)}< \frac{X_i}{Y_i^2} < \frac{1}{\sqrt{d_i}(1-\delta)}$
 for $s$ sufficiently large.

In Case (i), $\frac{X_i}{Y_i^2}$ is monotonic decreasing and bounded below
by $\frac{1}{\sqrt{d_i}(1-\delta)},$ so it tends to a finite limit which
must be at least $\frac{1}{\sqrt{d_i}(1-\delta)}> \frac{1}{\sqrt{d_i}}$,
contradicting the discussion above. Case (ii) is eliminated similarly.

The remaining possibility is that (iii) holds for all $\delta$;
hence $\frac{X_i}{Y_i^2} \rightarrow \frac{1}{\sqrt{d_i}}$ as claimed.
\end{proof}

We deduce that $g_i \dot{g_i}$ asymptotically approaches the constant
$\lambda_i/\sqrt{-C}$, and obtain the following theorem.

\begin{thm} \label{asymptotics}
The metric corresponding to our trajectory is, to leading order in $t$
as $t \rightarrow +\infty$,
\[
dt^2 + t \; {\sf h}_{\infty}
\]
where the homothety class of  ${\sf h}_{\infty}$ is that of the product
Einstein metric on $M_1 \times \cdots \times M_r$.
So our metric has an asymptotically paraboloid geometry.
\end{thm}

\begin{rmk}
The Bryant solitons on $\mathbb R^n$ (for $n>2$) also have
asymptotically paraboloid geometry.  On the other hand the complete
steady K\"ahler solitons considered in \cite{DW} are asymptotically circle
bundles of constant radius over paraboloids.

These two kinds of asymptotics may be viewed as the Ricci soliton
analogues of the Asymptotically Conical (AC) and Asymptotically Locally
Conical (ALC) conditions satisfied by many of the known complete
non-compact Ricci-flat metrics (see \cite{CGLP} for example).
\end{rmk}

\section{\bf Analysing the flow}

To check smoothness at the collapsing submanifold, we must now analyse
the trajectory as $s$ tends to $-\infty$. Recall that
$X_1 \rightarrow \beta = \frac{1}{\sqrt{d_i}},
Y_1 \rightarrow +\sqrt{1- \beta^2}$ and the remaining
variables tend to 0.

\begin{rmk} \label{asympLyap}
Observe from (\ref{eqnLyap}) that $\mathcal L$ (and hence $g_i^2 Y_i^2$)
tend to zero exponentially fast as $s$ tends to $-\infty$.
\end{rmk}

The following lemma is often useful.

\begin{lemma} \label{FHKlemma}
Suppose a function $F$ satisfies a differential equation
\begin{equation} \label{eqnFHK}
F^{\prime} = H F + K
\end{equation}
where $H, K$ are functions tending respectively to finite limits
$h,  k $ as $s$ tends to $-\infty$, where $h<0$ and $k \neq 0$.

Then either $\lim_{s \rightarrow -\infty} F(s) = -\frac{k}{h}$
or $F$ tends to $\infty$ or $-\infty$ as $s$ tends to $-\infty$. Moreover
in the case of infinite limit $F$ is monotonic for sufficiently large
negative $s$.
\end{lemma}

\begin{proof}
We give the proof for the case $k > 0$ below; obvious modifications
yield the proof for the case $k < 0$.

Let $0 < \delta < {\rm min} (-h, k)$, and choose $s^*(\delta)$ so that
for all $s \leq s^*(\delta)$ we have
\[
h - \delta < H(s) < h + \delta <0 \;\; : \;\;
0< k - \delta < K(s) < k + \delta.
\]

If $F(s_0) \leq \frac{k -\delta}{-h + \delta}$ for some $s_0 \leq s^*(\delta)$,
then $F^{\prime}(s_0) >0$. Hence these inequalities for $F,F^{\prime}$
actually hold for all $s \leq s_0$. So as $s$ tends to $-\infty$, either $F$
tends to $-\infty$ (monotonically on $(-\infty, s_0]$) or to a finite limit
$\xi <  \frac{k -\delta}{-h + \delta} < -\frac{k}{h}$. But in the latter
case $F^{\prime}$ tends to a nonzero limit, which is impossible.

If $F(s_0) \geq \frac{k + \delta}{-h - \delta},$ then we similarly see that
$F$ tends monotonically to $+\infty$ as $s$ tends to $-\infty$.

So if $F$ does not tend monotonically to $\pm \infty$, we see that
for all such $\delta$ we have $\frac{k -\delta}{-h + \delta}
< F(s) < \frac{k + \delta}{-h - \delta}$ on $(-\infty, s^*(\delta)]$.
Hence $F$ tends to $-\frac{k}{h}$ as $s$ tends to $-\infty$.
\end{proof}

\begin{lemma} \label{XYbounded}
The function $X_i / Y_i^2$ is positive, and remains bounded as $s$ tends to
$-\infty$.
\end{lemma}

\begin{proof} {   }

(i) {\bf Positivity}. Note that  $(\sum_{j=1}^{r} X_j^2) -1$ is bounded
above by $\mathcal L$ so it is negative. Now since $Y_i > 0$ for all $s$,
if $X_i \leq 0$ at some $s_0$, we see that $X_i^{\prime}$ is positive at $s_0$.
Hence $X_i$ is negative and $X_i^{\prime}$ positive on $(-\infty, s_0)$, which
contradicts the fact that $\lim_{s \rightarrow -\infty} X_i =0$ (if $i>1)$,
or $\beta = \frac{1}{\sqrt{d_1}}>0$ if $i=1$. We deduce that $X_i$ is positive
for all finite $s$.

(ii) {\bf Boundedness}. This is trivial if $i=1$, so in what follows we take
$i>1$. We know that $(\sum_{i=1}^{r} X_j^2) -1$ tends to $\beta^2 -1$
and $(\sum_{j=1}^{r} X_j^2) - \frac{X_i}{\sqrt{d_i}}$ tends to $\beta^2$ as $s$
tends to $-\infty$. Pick $s^*$ so that
$(\sum_{j=1}^{r} X_j^2) - 1 < \frac{1}{2} (\beta^2 -1)<0$ and
$(\sum_{j=1}^{r} X_j^2) - \frac{X_i}{\sqrt{d_i}} > \frac{1}{2} \beta^2 > 0$
for $s \leq s^*.$ Note that this implies that $Y_i^{\prime}$ is positive on
$(-\infty, s^*]$, since, as discussed in \S 3, we can take the $Y_i$ to be positive.

Suppose that $\frac{X_i}{Y_i^2} > \frac{2}{\sqrt{d_i}(1- \beta^2)}$ at some
$s_0 \leq s^*$. It follows from our choice of $s^*$ that $X_i^{\prime}$ is
negative at $s_0$. As remarked above, $Y_i$ and $Y_i^{\prime}$ are positive
on $(-\infty, s^*]$. It follows that these inequalities for $X_i/Y_i^2$ and
the derivatives of $X_i, Y_i$ actually hold on $(-\infty, s_0]$. But this
contradicts the fact that $X_i$ tends to zero as $s$ tends to
$-\infty$. So we have the desired bound on $X_i / Y_i^2$.
\end{proof}

\begin{prop} \label{XYlimit}
For $i > 1$, we have
\[
\lim_{s \rightarrow -\infty} \frac{X_i}{Y_i^2} = \frac{1}{\sqrt{d_i}(1 + \beta^2)}.
\]
\end{prop}

\begin{proof}
  The differential equation (\ref{eqnXY2}) is of the form
  (\ref{eqnFHK}), with $h = -(1 + \beta^2)$ and $k =
  \frac{1}{\sqrt{d_i}}$.  The desired result now follows from Lemma
  \ref{FHKlemma} and Lemma \ref{XYbounded}.
\end{proof}

\begin{cor}
As $s$ tends to $-\infty$, the arclength $t$ can be chosen to tend to
zero. Moreover, we have the following limiting values for $g_i(t)$ as
$t$ tends to $0$.
\[
g_1 (0) =0 \;\; : \;\; \dot{g}_{1}(0) = 1 \;\; :  \;\; \dot{g}_{i}(0) =0  \; (i>1).
\]
\end{cor}

\begin{proof}
The first statement follows from Remark \ref{asympLyap} and equation
(\ref{st}). The statement about $g_1$ comes from (\ref{def-tgi}) using
this Remark and the fact that $Y_1$ tends to a nonzero value as $s$ tends
to $-\infty$.

Since $\dot{g_i} = \sqrt{\lambda_i} (X_i/Y_i),$ the remaining limits
follow from Prop \ref{XYbounded}, the known limits of $X_i, Y_i$,
and the fact that $\lambda_1 = d_1 -1$.
\end{proof}

\begin{prop} \label{gi}
For $i > 1$, $g_i(0)$ is finite and nonzero.
\end{prop}

\begin{proof}
Observe that $\frac{d}{ds} (g_i^2) = 2g_i^2 X_i/\sqrt{d_i},$ which is positive
for all  $s$. So $g_i^2$ tends to a finite, nonnegative limit at $s=-\infty$.

To get positivity, we consider
\[
\frac{g_i^{\prime}}{g_i} =  \frac{X_i}{\sqrt{d_i}}
=\frac{1}{2 \sqrt{d_i}} \frac{X_i}{Y_i^2}
\frac{ (Y_i^2)^{\prime}}{(\sum_{j=1}^{r} X_j^2) - \frac{X_i}{\sqrt{d_i}}}.
\]
Integrating, using Prop. \ref{XYlimit}, and observing that the
denominator in the last factor tends to $\beta^2$,  we get
a bound
\[
\frac{g_i(s^*)}{g_i(s)} \leq \exp \left(\frac{Y_i^2(s^*)- Y_i^2(s)}
{2d_i (\beta^2 -\delta)(1 + \beta^2 - \delta)} \right)
\]
(for some positive $\delta$),
giving the desired positive lower bound on $g_i(s)$.
\end{proof}

\medskip

We shall next obtain some estimates that will be useful for
studying the second derivatives of $g_i$ and $u$.

\begin{lemma} \label{X1upper}
We have $X_1 < \beta$ for all finite $s$.
\end{lemma}

\begin{proof}
We can rewrite the equation (\ref{eqnX}) for $X_1$ as
\[
X_1^{\prime} = (X_1 - \beta) \left( {\mathcal L} - \sum_{j=1}^{r} Y_j^2 \right)
+ \beta \left({\mathcal L} - \sum_{j=2}^{r} Y_j^2 \right).
\]
The terms in the second and third brackets are negative for finite
$s$, since $\mathcal L$ is. Hence if $X_1 \geq \beta$ at some $s_0$ we
see $X_1^{\prime}$ is negative at $s_0$, and hence these inequalities
hold on $(-\infty, s_0]$. It follows that $X_1$ cannot tend to $\beta$
as $s$ tends to $-\infty$, a contradiction.
\end{proof}

Recalling that $\mathcal L$ is negative on our trajectory, this shows that
$\frac{X_1 - \beta}{\mathcal L}$ is positive.

\smallskip

We note next that $(X_1 - \beta)/{\mathcal L}$ satisfies the differential
equation
\[
\left(\frac{X_1 - \beta}{\mathcal L} \right)^{\prime} = \left(\frac{X_1 - \beta}{\mathcal L}
\right) \left( -1 - \sum_{j=1}^{r} X_j^2 \right) + \beta \left(
\frac{ \sum_{j=1}^{r} X_j^2 + Y_1^2 -1 }{\mathcal L} \right).
\]
Observe also that
\[
\rho := \lim_{s \rightarrow -\infty} \frac{ \sum_{j=1}^{r} X_j^2 + Y_1^2 -1}{\mathcal L}
\]
exists and is a finite number greater than $1$, because the
numerator is ${\mathcal L} - \sum_{j=2}^{r} Y_j^2$, and we know for
$j>1$ that $Y_j^2 / {\mathcal L}$ tends to a finite negative limit as
$s$ tends to $-\infty$ (as this is a negative constant times $g_j^{-2}$).

Now Lemma \ref{FHKlemma} shows that as $s$ tends to
$-\infty$, $\frac{X_1 - \beta}{\mathcal L}$ either tends to infinity
or tends to the positive limit $\frac{\beta \rho}{1  + \beta^2}$.
Moreover, because it cannot tend to $-\infty$, Lemma \ref{FHKlemma}
gives a positive lower bound on $\frac{X_1 - \beta}{\mathcal L}$.

\begin{rmk} \label{asymp} We can make some statements about the decay
  rates of our variables as $s$ tends to $-\infty$.  From (\ref{eqnY})
  we see that for $i > 1$ we have $Y_i \sim e^{\beta^2 s}$, hence, by
 Prop. \ref{XYlimit}, we also have $X_i \sim e^{2 \beta^2 s}$.
It follows from (\ref{eqnLyap}) that ${\mathcal L}
  \sim e^{2 \beta^2 s}$. Now the remarks in the previous paragraph show
  that $X_1 - \beta$ decays more slowly than $e^{(2 \beta^2 +
    \delta)s}$, for any positive $\delta$.
\end{rmk}

\begin{lemma} \label{XbLfin}
The quantity $\frac{X_1 - \beta}{\mathcal L}$ cannot tend to $\infty$ as $s$ tends
 to $-\infty$.
\end{lemma}

\begin{proof}
Observe that
\begin{equation} \label{Ybhat}
\frac{(Y_1 - \hat{\beta})^{\prime}}{{\mathcal L}^{\prime}} = \frac{Y_1}{2 \sum_{j=1}^{r} X_j^2}
\left( \frac{\sum_{j=2}^{r} X_j^2}{\mathcal L} + \frac{X_1 (X_1 - \beta)}{\mathcal L} \right).
\end{equation}
Now the term outside the bracket tends to a nonzero finite limit,
and the first term in the bracket tends to zero. So if
$\frac{X_1 - \beta}{\mathcal L}$ tends to $+\infty$, so does
$\frac{(Y_1 - \hat{\beta})^{\prime}}{{\mathcal L}^{\prime}}$, and hence,
by L'H\^{o}pital's rule, so does $\frac{Y_1 - \hat{\beta}}{\mathcal L}$.
But
\[
\frac{\sum_{j=1}^{r} X_j^2 + Y_1^2 - 1}{\mathcal L} =
\frac{ 2 \beta (X_1 - \beta) + 2 \hat{\beta} (Y_1 - \hat{\beta}) +
(X_1 - \beta)^2 + (Y_1 - \hat{\beta})^2  + \sum_{j=2}^{r} X_j^2}{\mathcal L}
\]
The discussion in Remark \ref{asymp} shows
 the first two terms are the dominant ones, so we deduce the expression
on the left-hand-side tends to $+\infty$, which we saw above is false.
\end{proof}

We now have
\begin{prop} \label{XbLlimit}
\[
\lim_{s \rightarrow -\infty} \left( \frac{X_1 - \beta}{\mathcal L} \right) =
\frac{\beta \rho}{1 + \beta^2}.
\]
\end{prop}

\begin{prop} \label{Lasymp}
As $s$ tends to $-\infty$, $e^{-2 \beta^2 s} {\mathcal L}$ tends to a finite
negative limit.
\end{prop}

\begin{proof}
Letting $F = e^{-2 \beta^2 s} {\mathcal L}$, (so that $F$ is negative),
 we see that
\[
F^{\prime} = 2F ( (\sum_{j=1}^{r} X_j^2) - \beta^2 ).
\]
As in the proof of Lemma \ref{XbLfin}, we see that the  dominant term in
the bracket is $2 \beta (X_1 - \beta),$ which is negative by Lemma \ref{X1upper}.
So $F$ is monotonic increasing for large negative $s$.

The proposition is proved if we can show that $F$ is bounded below near
$s=-\infty$. This follows by estimating the integral of the right-hand side
of the above equation over an interval $(-\infty, s^*]$ on which we have
bounds of the form
$$ |X_1 - \beta| \leq A\exp(2(\beta^2 -\delta)s), \ \ \ \ \ \  |X_j(s)| \leq C_j \exp(2(\beta^2 -\delta)s)$$
$$ |X_1 + \beta| \leq 2\beta + \delta, \ \ \ \ \ \
      |{\mathcal L}| \leq B \exp(2(\beta^2 - \delta)s)$$
where $\delta, A, B, C_j, 2 \leq j \leq r$ are appropriate positive constants
(cf Remark \ref{asymp}).
\end{proof}

We can now check the second derivatives of $g_i$ at $t=0$. Recall that
from (\ref{gidotdot}) we have
\[
\ddot{g}_i = \frac{\lambda_i}{g_i Y_i^2} (X_i^2 + Y_i^2 - \sqrt{d_i} X_i).
\]
It is now clear from Prop \ref{XYlimit}
and Prop \ref{gi} that for $i >1$, $\ddot{g_i}$ tends
to a finite limit as $t$ tends to zero.

If $i=1$, we rewrite this expression, using (\ref{def-Yi})-(\ref{cons2}), as
\[
\ddot{g_1} = \frac{\lambda_1}{Y_1} \left(\frac{1 - \sqrt{d_1}X_1}{g_1 Y_1}\right)
+ g_1 \left( \frac{C}{d_1} - \sum_{j=2}^{r}
\frac{d_j \lambda_j}{d_1} \frac{1}{g_j^2} \left(\frac{X_j^2}{Y_j^2} + 1\right) \right).
\]
The quantities in the bracket after $g_1$ tend to finite limits, so
$g_1$ times that bracket tends to zero.
As $Y_1$ tends to $\hat{\beta} \neq 0$, we see from
Prop \ref{XbLlimit} and (\ref{cons2}) that the first term also tends to zero.

We have therefore shown the metric is $C^2$.

Similarly we can study the potential $u$. From the relation (\ref{def-u})
we obtain
\begin{equation} \label{u1}
\dot{u} = \sqrt{\frac{C}{\mathcal L}} \left(\left(\sum_{j=1}^{r} \sqrt{d_j} X_j\right) -1 \right)
= \sqrt{d_1}\left(\frac{ (X_1 - \beta)}{\sqrt{{\mathcal L}/C}}\right)
+ \sum_{i=2}^{r} \sqrt{d_i}\left(\frac{X_i}{Y_i^2}\right) \left(\frac{Y_i^2}{\sqrt{{\mathcal L}/C}}\right).
\end{equation}
By Prop \ref{XYlimit}, Remark \ref{asymp}, and Prop \ref{XbLlimit}, it follows
that $\dot{u}$ tends to zero as $t$ tends to $0$. Next, by integrating (\ref{def-u})
we get
\begin{equation} \label{u0}
e^{u(0)} = ({\rm positive \; constant })
\prod_{i=2}^{r} \, g_i(0)^{d_i} \left(\lim_{s \rightarrow
-\infty} e^{-2 \beta^2 s} \frac{\mathcal L}{C}\right)^{\frac{d_1}{2}},
\end{equation}
which is finite by Prop \ref{Lasymp}. If we differentiate (\ref{def-u}) and
use (\ref{eqnLyap}) we get
\begin{equation} \label{u2}
\ddot{u} = \sum_{i=1}^{r} \, d_i \frac{\ddot{g_i}}{g_i}
= \sum_{i=1}^{r} \frac{\lambda_i d_i}{g_i^2 Y_i^2} \left(X_i^2 + Y_i^2 -\sqrt{d_i}X_i \right).
\end{equation}
The right-hand side tends to a finite limit (as $t$ tends to $0$) by
Props \ref{gi}, \ref{XbLlimit}, the relation (\ref{cons2}) and the discussion
after Prop \ref{Lasymp}.

For the third derivatives, we calculate
\[
\frac{d^3 g_i}{dt^3} = \sqrt{\frac{C}{{\mathcal L}}} \frac{\lambda_i}{g_i}
\left( \frac{X_i}{Y_i^2} \left(
 -3X_i + \frac{X_i^2}{\sqrt{d_i}} + \sqrt{d_i} + \sqrt{d_i}
\sum_{j=1}^{r} X_j^2 \right)
+ \frac{X_i}{\sqrt{d_i}} -1 \right).
\]
If $i > 1$, we know that the terms in $\frac{X_i}{\sqrt{\mathcal L}},
\frac{X_i^2}{Y_i^2 \sqrt{\mathcal L}},
\frac{X_i^3}{Y_i^2 \sqrt{\mathcal L}}$ tend to zero (cf Remark \ref{asymp}).
Our task thus reduces to showing that
\[
\frac{1}{\sqrt{\mathcal L}} \left(\frac{X_i}{Y_i^2} (1 + \sum_{j=1}^{r} X_j^2 )
 -\frac{1}{\sqrt{d_i}} \right)
\]
tends to zero.

Now, $1 + \sum_{j=1}^{r} X_j^2 = 1 + \beta^2$ modulo terms approaching
zero at least as fast as $\mathcal L$, so we just have to check that
\[
\lim_{s \rightarrow -\infty}
  \frac{1}{\sqrt{\mathcal L}} \left(\frac{X_i}{Y_i^2} -
  \frac{1}{\sqrt{d_i}(1 + \beta^2)}\right) =0.
\]
In fact we shall show the stronger statement that
\[
Q_i := \frac{1}{\mathcal L} \left( \frac{X_i}{Y_i^2}
- \frac{1}{\sqrt{d_i}(1 + \beta^2)}\right)
\]
tends to a finite limit.
We find that $Q_i$ satisfies the equation
$$ Q_i^{\prime} = -\left(1+ 3 \sum_{j=1}^r X_j^2 \right) Q_i
  + \frac{1}{(\sqrt{d_i}(1+ \beta^2))\mathcal L} \left(2(1+\beta^2) \frac{X_i^2}{Y_i^2}
    + \beta^2 - \sum_{j=1}^r X_j^2 \right). $$
By arguments similar to those above, one sees that the second term on
the right-hand side tends to a finite negative limit
as $s \rightarrow -\infty$ since $X_1 < \beta$ and ${\mathcal L} < 0.$
So the hypotheses of Lemma \ref{FHKlemma} are satisfied, and $Q_i$ either
tends to a finite limit or to $+\infty$ or $-\infty$.  But
\[
\frac{\left(\frac{X_i}{Y_i^2} - \frac{1}{\sqrt{d_i}(1 + \beta^2)}\right)^{\prime}}
{{\mathcal L}^{\prime}}
=
\frac{\left(\frac{X_i}{Y_i^2} - \frac{1}{\sqrt{d_i}(1 + \beta^2)}\right)}
{\mathcal L} \left( \frac{-(1 + \beta^2)}{2 \sum_{j=1}^{r} X_j^2} \right)
+ R_i
\]
where $R_i$ tends to a finite limit. So
if the limit of $Q_i$ is infinite, L'H\^{o}pital's rule gives a
contradiction, as the term in the final bracket is negative.

For $i=1$, by (\ref{def-tgi}) and $Y_1(0)=\hat{\beta},$ it is enough to check that
\[
\frac{ -3 (\frac{X_1}{Y_1})^2 + \beta \frac{X_1^3}{Y_1^2} + \beta X_1 -1
+  \frac{X_1}{\beta Y_1^2} (1 + \sum_{j=1}^{r} X_j^2)  }
{\mathcal L}
\]
has a finite limit.
Recall that $1 + \sum_{j=1}^{r} X_j^2 =1 + \beta^2$ modulo terms
which have a finite limit when divided by $\mathcal L$. Similarly
$X_1^k = \beta^k$ modulo such terms. So we are left with checking
that the limit of $\frac{Y_1^2 - \hat{\beta}^2}{\mathcal L}$ is finite, which follows
on applying L'H\^{o}pital's rule, equation (\ref{Ybhat}) and Prop. \ref{XbLlimit}.

\medskip

We have shown that the metric is $C^3$ and so by the discussion on regularity
near the beginning of \S 3, the soliton is smooth.

\begin{thm} \label{mainthm}
Let $M_2, \ldots, M_r$ be compact Einstein manifolds with positive scalar
curvature. For $d_1 > 1$ there is an $r-1$ parameter family of complete smooth
steady Ricci solitons on the trivial  rank $d_1 + 1$ vector bundle
over $M_2 \times \ldots \times M_r$.  \qed
\end{thm}

The examples of Ivey and Bryant have nonnegative Ricci curvature. This is
also true for our more general examples.

\begin{prop} \label{ricci}
The soliton metrics have nonnegative Ricci curvature.
\end{prop}

\begin{proof}
It is enough to show that the Ricci curvature is positive on the complement
of the submanifold at $t=0$, i.e., on the finite part of the trajectory.
The Hessian of $u$ is given by $\ddot{u}$ evaluated on directions normal
to the hypersurface and by $\frac{\dot{u} \dot{g_i}}{g_i}$ evaluated on
directions tangent to $M_i$. From the soliton equation (\ref{gradRS})
(with $\epsilon = 0$), the formulae (\ref{u0}), (\ref{u1}), (\ref{u2}),
(\ref{cons2}) above for $\dot{u}, \ddot{u}$ and the fact, proved in Prop (\ref{gi}),
that $\frac{\dot{g}_{i}}{g_i}$ is positive, it is enough to check that
\[
{\mathcal H}= \sum_{i=1}^{r} \sqrt{d_i} X_i,
\]
introduced in Remark \ref{einstein}, satisfies ${\mathcal H} < 1$ and
${\mathcal L} + 1 - {\mathcal H} < 0$ for all $s$.

Now it is easy to check that we have equations
\[
({\mathcal H}-1)^{\prime} = ({\mathcal H}-1) ( \sum_{j=1}^{r} X_j^2  -1 )
 + \mathcal L
\]
and
\[
({\mathcal L} + 1 - {\mathcal H})^{\prime} =
({\mathcal L} + 1 - {\mathcal H}) (\sum_{j=1}^{r}
X_j^2 -1 ) + {\mathcal L} (\sum_{j=1}^{r} X_j^2).
 \]
 Moreover $\sum_{j=1} X_j^2 - 1$ and $\mathcal L$ are negative.  So if
 ${\mathcal H} \geq 1$ at $s_0$ then ${\mathcal H}^{\prime} <0$ and
 ${\mathcal H} > 1$ on $(-\infty, s_0)$, contradicting the fact that
 $\mathcal H$ tends to $1$ as $s$ tends to $-\infty$. Similarly if
 ${\mathcal L} +1 -{\mathcal H}$ is non-negative at $s_0$, then we see it
 is positive with negative derivative on $(-\infty, s_0)$,  contradicting
 the fact that ${\mathcal L} + 1-{\mathcal H}$ tends to zero.
\end{proof}

\begin{rmk} \label{sectcurv}
For a multiply-warped product $I \times M_1 \times \cdots \times M_r$ with
metric of the form (\ref{metric}) it is easy to compute the sectional curvatures,
e.g., by considering it as a Riemannian submersion over $I$. If $U, V$ are
respectively tangent to  $M_i$ and $M_j$, one has
\begin{eqnarray*}
 K(U \wedge \frac{\partial}{\partial t}) & = & -\frac{\ddot{g}_i}{g_i}  \\
 K(U \wedge V) &=& -\frac{\dot{g}_i \dot{g}_j}{g_i g_j}, \, \, \, \, i \neq j, \\
 K(U \wedge V) &=& \frac{1}{g_i^2} \left(K_{h_i}(U \wedge V) - \dot{g_i}^2  \right), \, \, i=j,
\end{eqnarray*}
where $K_{h_i}$ denotes the sectional curvature of $(M_i, h_i)$.

It now follows from the asymptotics described in Theorem \ref{asymptotics}
that if $r > 1$ there are always $2$-planes with negative sectional curvature.
If $r=1$, we are in the case of the Bryant solitons, which are known to have positive
curvature ([Bry] or [Cetc Lemma 1.37]). The above formulas also show
that the sectional curvatures decay like $t^{-1}$ as $t$ tends
to $+\infty$. Recall that the {\em asymptotic scalar curvature ratio} of a
complete, non-compact Riemannian manifold is defined as
$\limsup_{d \rightarrow +\infty}  R d^2$, where $R$ is the scalar curvature
and $d$ is the distance from a fixed origin in the manifold. Since $t$ is the
geodesic distance in our examples, it follows that their asymptotic scalar
curvature ratios are all $+\infty$.
\end{rmk}

\end{document}